\documentclass{article}
\usepackage{geometry,amsmath,amssymb,bbm}
\usepackage{enumitem}
\usepackage{bbm}
\usepackage{hyperref}
\newenvironment{proof}{\noindent\textbf{Proof\ }}{\hspace*{\fill}$\Box$\medskip}
\newtheorem{lemma}{Lemma}
\newtheorem{proposition}{Proposition}

\newtheorem{definition}{Definition}
\newtheorem{remark}{Remark}

\newtheorem{theorem}{Theorem}

\begin{document}

\title{Iwahori Spherical Whittaker Functions for Steinberg Representations} \author{Ehud Moshe Baruch and Markos Karameris}\maketitle
\begin{abstract}

 Let $G$ be a split reductive group over a $p$-adic field $F$ and let $(\pi_{St},V)$ be a Steinberg representation of $G$ with trivial central character. It is known that the space of Iwahori fixed vectors in $V$ is one-dimensional. The Iwahori Hecke algebra acts on this space via a character. We compute in full the Whittaker function associated to this vector using the Hecke algebra action. We use the explicit form of this Whittaker function to compute certain Rankin-Selberg integrals involving Steinberg representations in the case of $GL_n(F)$.
\end{abstract}

\section{Introduction and Notations}
Reeder \cite{reeder} considered certain Iwahori fixed vectors in induced representations and computed their Whittaker values on diagonal elements. He furthermore acknowledged the challenge of computing these values for a general element. In this paper we give a complete formula of an Iwahori fixed Whittaker function in a Steinberg representation. In \cite{Bump}, Bump, Brubaker, Buciumas, Henrik, and Gustaffson provide a formula for computing the values of Iwahori fixed vectors in induced representations. Our vector is a linear combination of such vectors, however it is not clear how to obtain our result from their general formulas. We explicitly derive the formula for a Whittaker function of an Iwahori fixed vector in the Steinberg representation giving all values of this function.

Let $F$ be a non-archimedean local field and $|.|$ the standard absolute value on $F$. We denote by $\mathfrak{p}$ the maximal ideal, generated by the uniformizer $\varpi$, $\mathcal{O}$ the ring of integers and $\mathfrak{f}=\mathcal{O}/\mathfrak{p}$ the residue field with $q$ elements. We also consider the connected, split reductive group of $F$-rational points $G=\mathcal{G}(F)$ where $\mathcal{G}$ is the corresponding group-valued functor. We furthermore let $\mathcal{Z}$ denote the center of $\mathcal{G}$ and $Z=\mathcal{Z}(F)$ the center of $G$. We also have the following subgroups inside $G$:
 the maximal compact subgroup $K=\mathcal{G}(\mathcal{O})$,
 the Borel subgroup $B$,
 the maximal unipotent $N$,
 the restriction $N_{\mathcal{O}}=N\cap K$,
 the maximal torus $T$,
 the Weyl group of permutations $\mathbf{W}=N_G(T)/T$ \footnote{For each $\mathbf{w}\in\mathbf{W}$ we pick a representative $w$ in $K\cap N_G(T)$.} generated by the simple reflections $\mathbf{s}_{\alpha}$ for the simple roots $\alpha$ and $\mathbf{w_0}$ the longest Weyl element and 
 the Iwahori subgroup $J$, which is the preimage of the Borel under the canonical projection $K\to\mathcal{G}(\mathfrak{f})$.
 Let $X^*(T)$ be the group of rational characters of $T$ and $X_*(T)$ the group of rational cocharacters $\lambda:F\to T$. The mapping $X_*(T)\to T/T^0$, where $T^0=T\cap K$, with $\lambda\to\varpi^{\lambda}$ and $$\varpi^{\lambda}=\lambda(\varpi)$$ induces an isomorphism $X_*(T)\cong T/T^0$. For each $d\in T/T^0$ we denote by $\lambda_d$ the corresponding cocharacter in $X_*(T)$. We have the usual perfect pairing $\langle.,.\rangle$ between the groups $X^*(T)$ and $X_*(T)$ which is defined by $\chi(\lambda(\varpi))=\varpi^{\langle\chi,\lambda\rangle}$ where $\chi\in X^*(T)$ and $\lambda\in X_*(T)$.
 Denote by $\Phi$ the associated root system with positive simple roots $\Delta$. We also let $x_{\alpha}(t)$ be the one-parameter subgroup corresponding to $\alpha\in\Phi$. Then $N=\prod\limits_{\alpha\in\Phi^+}x_{\alpha}(F)$ and, moreover, the elements $d\in T$ act on the roots by conjugation as \begin{equation}dx_{\alpha}(t)d^{-1}=x_{\alpha}(\alpha(d)t)=x_{\alpha}(st \varpi^{<\alpha, \lambda_d>} )\label{actroot}\end{equation}  where $s\in\mathcal{O}^{\times}$ and $\alpha(d)=\varpi^{\langle\alpha,\lambda_d\rangle}\mod\mathcal{O}^{\times}$. The Weyl element representatives can be chosen so that $wx_{\alpha}(t)w^{-1}=x_{w\alpha}(ct)$ where $c=\pm1$.
 \begin{remark}\label{rem1}
     It is easy to see that $\langle\alpha,\lambda_{wdw^{-1}}\rangle=\langle w^{-1}\alpha,\lambda_d \rangle$ by explicitly computing the expression $wdw^{-1}x_{\alpha}(t)wd^{-1}w^{-1}=wx_{w^{-1}\alpha}(c_1t\varpi^{\langle w^{-1}\alpha,\lambda_d\rangle})w^{-1}=x_{\alpha}(c_1c_2t\varpi^{\langle w^{-1}\alpha,\lambda_d\rangle})$ where $c_1,c_2\in\mathcal{O}^{\times}$.
 \end{remark}
 In addition to the above, let $\hat{G}$ be the Langlands dual of $G$ and $\hat{T}$ the dual torus. There is a natural isomorphism between $X_*(T)$ and $X^*(\hat{T})$ which identifies each $\lambda\in X_*(T)$ with an element in the coweight lattice $X^*(\hat{T})$. More explicitly, letting $\hat{T}=Hom_{\mathbb{Z}}(T/T^0,\mathbb{C})$, we can parametrize each element of $\hat{T}$ with a $(\mathbb{C}^{\times})^n$ tuple $z=(z_1,...,z_n)$ and the associated homorphism $\tau_z(t)=\prod\limits_{i=1}^nz_i^{\lambda_i}$ where $t\in\varpi^{\lambda}T^0$ and $\lambda$ is identified with $(\lambda_1,...,\lambda_n)\in\mathbb{Z}^n$ under the natural isomorphism $X_*(T)\cong \mathbb{Z}^n$. In this way we obtain a perfect pairing between $X^*(T)$ and $X^*(\hat{T})$.

\section{The Iwahori-Bruhat decomposition}

Let $\mathcal{H}(G,J)$ be the Hecke algebra of $J$ bi-invariant and compactly supported functions on $G$. $\mathcal{H}(G,J)$ acts on spaces of $J$ invariant vectors in a representation of $G$ and 
in particular on right $J$ invariant functions on $G$. We denote by
$X_g\in \mathcal{H}(G,J)$ the characteristic function of $JgJ$. Then the left action of the Hecke algebra on a $J$ invariant function $F:G\to\mathbbm{C}$ is given by $X_g(F)(x)=\int_{JgJ}F(xh)dh$ (see 1.2 of \cite{bor}) where $\mu=dh$ is a normalized Haar measure for which $\mu(J)=\int_Jdh=1$. 
In fact, we can compute this action explicitly:
\begin{lemma} \label{cosum1}
    For any $g\in G$ with coset decomposition $JgJ=\overset{m}{\underset{i=1}{\bigcup}}\gamma_iJ$ and $f$ a $J$ invariant function. Then $$X_{g}(f)(x)=\sum\limits_{i=1}^mf(x\gamma_i).$$ We note that $m=\mu(JgJ)=\#|J\backslash JgJ|$.
\end{lemma}

We begin by recalling the Iwahori factorization of $J$:
$$J=N_{\mathcal{O}}T^0N_{\mathfrak{p}}^-$$ where $N_{\mathcal{O}}=N\cap K$ and $N_{\mathfrak{p}}^-=N^-\cap J$ where $N^-$ is the opposite unipotent subgroup. In terms of root groups we have $N^-=\prod\limits_{\beta\in\Phi^-}x_{\beta}(F)$ and thus $N_{\mathfrak{p}}^-=\prod\limits_{\beta\in\Phi^-}x_{\beta}(\varpi t_{\beta})$ where $t_{\beta}\in\mathcal{O}$.

The results on coset decompositions below are partially known from Iwahori-Matsumoto \cite{iwmats}:

\begin{lemma}\label{cosrep}
    For any $w\in\mathbf{W}$ with $\ell(w)=k$: $$JwJ=\underset{(t_{\alpha})\in(\mathcal{O}/\mathfrak{p})^k}{\bigcup}\prod\limits_{\alpha\in\Phi|w^{-1}\alpha\in\Phi^-}x_{\alpha}(t_{\alpha})wJ.$$
\end{lemma}
\begin{proof}
By the Iwahori factorization: $J=N_{\mathcal{O}}T^0N_{\mathfrak{p}}^-$ and $N_{\mathfrak{p}}^-=\prod\limits_{\beta\in\Phi^-}x_{\beta}(\varpi t_{\beta})$ where $t_{\beta}\in\mathcal{O}$. But $w^{-1}T^0N_{\mathfrak{p}^-}w\subset J$. Basic Chevalley group theory allows us to write $N_{\mathcal{O}}=\mathcal{N}_w^{-}\mathcal{N}_w^{+}$ where $\mathcal{N}_w^{+}=\langle x_{\alpha}(\mathcal{O})\rangle_{\alpha\in\Delta:w^{-1}\alpha>0}$ and $\mathcal{N}_w^{-}=\langle x_{\alpha}(\mathcal{O})\rangle_{\alpha\in\Delta:w^{-1}\alpha<0}$. Since
$w^{-1}x_{\alpha}(t_{\alpha})w=x_{w^{-1}\alpha}(ct_{\alpha})\in J$ if and only if $w^{-1}\alpha\in\Phi^+$ or $w^{-1}\alpha\in\Phi^-$ and $t_{\alpha}\in\mathfrak{p}$ it follows that $w^{-1}\mathcal{N}_w^{-}w\subseteq J$. Finally observing that $\#\{\alpha\in\Phi|w^{-1}\alpha\in\Phi^-\}=\ell(w^{-1})=\ell(w)$ completes the proof.
\end{proof}\\
Using Lemma \ref{cosrep} and the Iwasawa decomposition we obtain the refined Bruhat decomposition of $G$ relative to the Iwahori subgroup:
 \begin{lemma}[Bruhat-Iwahori decomposition \cite{iwmats}]\label{decomp} 
 There exists a double coset decomposition of $G$ given by:
    $$G=\underset{\lambda\in X_*(T)}{\bigcup}\underset{w\in\mathbf{W}}{\bigcup}N\varpi^{\lambda}wJ$$.
 \end{lemma}

 \section{Iwahori spherical vectors of Steinberg Representations}
 Let $\chi$ be a smooth character of $G$ that is trivial on $K$ and $\delta_B$ be the modular character of $B$. We denote by $I(\chi)=Ind_B^{G}(\delta_B^{1/2}\chi)$ the principal series representation formed by normalized induction from the smooth multiplicative character $\chi$. (see also section 1. of \cite{shal}):
 \begin{definition}
     The Steinberg representation corresponding to an unramified character $\chi$ is the unique irreducible subrepresentation of $I(\chi)$. We denote this representation by ${St}^{\chi}$.
 \end{definition}
 \begin{remark} (see p. 70 of \cite{cassnotes})
     The Steinberg representation can equivalently be defined as the unique irreducible quotient of the representation $I(\chi^{-1}\delta_B^{-1})$.
 \end{remark}

For any split reductive group $G$, the Casselman basis of $I(\chi)^J$ which consists of characteristic functions of the form:
$$f_w(g)=\begin{cases}
        \chi(b)\delta_B(b) & \text{if } g=bwj, b\in B,j\in J,   \\
        0 & \text{otherwise} 
    \end{cases}.
$$
 From Lemma 8.1.2 of \cite{cassnotes} it follows that $dim((St^{\chi})^J)=1$. Then due to Borel (see 5.7, \cite{bor}) we have the well known result that the characteristic functions corresponding to the simple reflections $X_{s_{\alpha}}$ act on $(St^{\chi})^J$ with the scalar $-1$ (independently of $\chi$) also called the sign character of the Weyl group. This enables us to determine the corresponding eigenvector. Let $\mathcal{H}_{\mathbf{W}}$ denote the Weyl group component of $\mathcal{H}(G,J)$.
\begin{proposition} \label{fixvec}
    The Iwahori fixed part of the principal series representation $I(\chi)^J$ contains, up to scalar multiplication, exactly one eigenvector of $\mathcal{H}_{\mathbf{W}}$ affording the sign character given by $\phi_G^-(g)=\sum\limits_{w\in W}(-q)^{-\ell(w)}f_w(g)$(see also p.23 of Reeder \cite{reeder}).
\end{proposition}
\begin{proof}
    Consider the complex character of the affine part of $\mathcal{H}(G,J)$ given by $sgn:X_{s_{\alpha}}\to -1, \forall\alpha\in\Delta$. Let $\phi(g)=\sum\limits_{w\in W}c_wf_w(g)$ be an eigenvector of $sgn$ inside $I(\chi)^J$. From Lemma \ref{cosrep} we have that for any $w\in\bf{W}$: $X_w(\phi_G^-)(1)=\rho^-(X_w)\phi_G^-(1)=(-1)^{\ell(w)}c_1$. We observe that for any $w'\in \bf{W}$: $$X_w(f_{w'})(1)=\sum\limits_{(t_{\alpha})\in(\mathcal{O}_F/\mathfrak{p})^{\ell(w)}}f_{w'}(\prod\limits_{\alpha\in\Phi|w^{-1}\alpha\in\Phi^-}x_{\alpha}(t_{\alpha})w)=q^{\ell(w)}\delta_{w,w'}$$ which implies $$X_w(\phi_G^-)(1)=q^{\ell(w)}c_w$$ and thus we obtain: $$c_w=(-q)^{-\ell(w)}c_1$$ respectively.
    It is easy to verify that this is an eigenvector of $\mathcal{H}_{\mathbf{W}}$ indeed.
\end{proof} \par 
Since $\phi_G^-$ is the unique eigenvector inside the induced space affording the $-1$ eigenvalue for every simple root, $\phi_G^-$ is the unique Iwahori fixed vector of the Steinberg representation. Because $\phi_G^-$ is an eigenvector of the full Iwahori Hecke Algebra $\mathcal{H}(G,J)$, the Hecke algebra acts on it by a character which extends the sign character and is dependent on $\chi$. This character is defined by the equality $X_g(\phi_G^-)=\rho_\chi(X_g)\phi_G^-$. The uniqueness of the character is also stated in section 1. of \cite{shal}.

\section{Iwahori fixed Whittaker functions}
We compute the Whittaker function associated to an Iwahori fixed vector of a Steinberg representation. For an unramified character $\psi:F\to\mathbbm{C}^{\times}$ we extend it to a character $\psi:N \rightarrow \mathbbm{C}^{\times}$ in the following way: let $n=\prod\limits_{\alpha\in\Phi^+}x_{\alpha}(t_{\alpha})$, then 
$\psi(n)=\sum\limits_{\alpha\in\Delta^+}\psi(t_{\alpha})$. Let $\mathcal{W}(\pi_{St^{\chi}},\psi)$ be the Whittaker model of $St^{\chi}$ with character $\psi$. It is a space of smooth functions $W:G\to\mathbb{C}$ satisfying my $W(ng)=\psi(n)W(g)$ for any $n\in N$ and $g\in G$, such that this space equipped with the right action of $G$ is isomorphic to $\pi_{St^{\chi}}$. It is well known that this space is unique.  We are interested in $\mathcal{W}(\pi_{St^{\chi}},\psi)^J$ which,is $1$ dimensional by Lemma 8.1.2. of \cite{cassnotes}. Let $W\in \mathcal{W}(\pi_{St^{\chi}},\psi)^J$ with $W\not\equiv 0$, then $W(gj)=W(g),\forall g\in G$ and $j\in J$ and $(F*W)(g)=\rho_{\chi}(F)W(g)$ with $F\in \mathcal{H}(G,J)$ and $\rho_{\chi}$ the character of $\mathcal{H}(G,J)$ as above. 
From Lemma \ref{decomp} it is clear that the values of $W$ are completely determined by the elements in $T\mathbf{W}$ and more specifically by the elements of $(T/T^0)\mathbf{W}$ since $T^0\subseteq J$. This means that we can restrict ourselves to only computing the value of $W$ on $\varpi^{\lambda_d}w$. In what follows we always assume $W$ is as above. 
\par 
We first assume that $\mathcal{G}$ has connected center $\mathcal{Z}$ (ex. $GL_n$ or $GSp_{2n}$). This property implies that the set of fundamental coweights generates the coweight lattice $X^*(\hat{T})$ or equivalently that there exists a basis $\{\lambda^{\beta}|\beta\in\Delta\}$ of $X^*(\hat{T})$ such that $\langle\alpha,\lambda^{\beta}\rangle=\delta_{\alpha,\beta}$ for every $\alpha\in\Delta$ (see the end of the introduction).  Such $\lambda^{\beta}$ exist for a group with a connected center (ex. the group $SL_n(F)$ with finite center $\mu_n$ does not contain any of the $s_{\alpha}$-dominant weights).

To compute the support of $W$ we use the following result, a variant of which appears in \cite{Bump}:
\begin{lemma} \label{supp}
    Let $d\in T$, then $W(dw)=0$ unless $\langle\alpha,\lambda_d\rangle\geq \begin{cases}
        0 & \text{if $w^{-1}\alpha\in \Phi^+$} \\
        -1 & \text{if $w^{-1}\alpha\in \Phi^-$} 
    \end{cases}$ for all $\alpha\in\Delta$. \\ A weight $\lambda_d$ satisfying this condition and the corresponding element $d\in T$ will both be called $w$-dominant. 
\end{lemma}
\begin{proof}Assume that there exist $\alpha \in \Delta$ such that $w^{-1}(\alpha) \in \Phi^{+}$ and $\langle\alpha, \lambda_d\rangle  < 0$ or that there exist  $\alpha \in \Delta$ such that $w^{-1}(\alpha) \in \Phi^{-}$ and $\langle\alpha, \lambda_d\rangle  < -1$. Let $s\in F$ be such that $|s|=q$ and $\psi(x_{\alpha}(s))\ne 1$. Then: $\psi(x_{\alpha}(s))W(dw)=W(x_{\alpha}(s)dw)=W(dd^{-1}x_{\alpha}(s)dw)=W(dww^{-1}x_{\alpha}(\varpi^{-\langle\alpha,\lambda_d\rangle} s)w)= W(dwx_{w^{-1}\alpha}(\varpi^{-\langle\alpha,\lambda_d\rangle} s_1))$ where $|s_1|=|s|$. Since $x_{w^{-1}\alpha}(\varpi^{-\langle\alpha,\lambda_d\rangle} s_1)\in J$ in each case, we then have that $\psi(x_{\alpha}(s))W(dw)=W(dwx_{w^{-1}\alpha}(\varpi^{-\langle\alpha,\lambda_d\rangle} s_1))=W(dw)$ and thus it follows that $W(dw)=0$. 
\end{proof}

\begin{definition}\label{w-dom}
        We denote the semigroup of $w$-dominant elements in $T$ by $T_w^+$. Forthermore, for a Weyl element $w$ we define $\lambda_w=-\sum\limits_{\underset{w^{-1}\alpha\in\Phi^-}{\alpha\in\Delta}}\lambda^{\alpha}$ and define the element $d_w\in T/T^0$ to be the image of $\lambda_w$ evaluated at $\varpi$ i.e. $d_w=\lambda_w(\varpi)$. It is clear that $\lambda_{d_w}=\lambda_w$.
\end{definition}

\begin{remark} \label{domd} It is straightforward to verify that $d_w$ is the unique element in $T/T^0$ satisfying the condition
 $\langle\alpha,\lambda_{d_w}\rangle= \begin{cases}
        0 & \text{if $w^{-1}\alpha\in \Phi^+$} \\
        -1 & \text{if $w^{-1}\alpha\in \Phi^-$} 
    \end{cases},\forall\alpha\in\Delta$.
\end{remark}

\begin{remark}
  For any $d\in T$ we have $W(dw)=0$ unless $d\in T_w^+$.
\end{remark}

 To ease notation we let $T^+=T_1^+$. It is clear from the defining relation of Lemma \ref{supp} that $T_1^+\subseteq T_w^+,$ for all $ w\in\mathbf{W}$, and in particular $T_w^+=d_wT^+$.

\begin{lemma} \label{permlem}
         For a Weyl element $w$ and $d\in T^+$ the following relation is satisfied:
 $$W(dw)=(-q)^{-\ell(w)}W(d).$$
\end{lemma}
\begin{proof}
Inductively it follows that: $X_{w}(W)(g)=(-1)^{\ell(w)}W(g)$.
From Lemma \ref{cosrep} however we have: $JwJ=\underset{(t_1,\dots t_{\ell(w)})\in \mathcal({O}_F/\mathfrak{p})^{\ell(w)}}{\bigcup}\prod\limits_{j=1}^{\ell(w)}x_{\alpha_j}(t_{j})wJ$ with $\prod\limits_{j=1}^{\ell(w)}x_{\alpha_j}(t_j)\in N_{\mathcal{O}}$ and $\mu(JwJ)=|JwJ/J|=q^{\ell(w)}$. Lemma \ref{cosum1} implies that  $X_{w}(W)(d)=\sum\limits_{(t_1,\dots t_{\ell(w)})\in \mathcal({O}_F/\mathfrak{p})^{\ell(w)}}W(d\prod\limits_{j=1}^{\ell(w)}x_{\alpha_j}(t_j)w)=q^{\ell(w)}W(dw)$ and since $dx_{\alpha_i}(t_j)d^{-1}\in N_{\mathcal{O}}^+$ for any $d\in T^+$, we obtain that $W(dw)=(-q)^{-\ell(w)}W(d)$.
\end{proof}
\begin{lemma} \label{dreps}
    For any $d\in T^+$ we have $JdJ\subseteq  N_{\mathcal{O}}^+dJ$ hence $JdJ=\bigcup\limits_{\ell\in L}n_{\ell}dJ$ for a finite set of coset representatives $n_{\ell}\in N_{\mathcal{O}}^{+}$. Notice that $|L|=\mu(JdJ)$.
\end{lemma}
\begin{proof}
    For any $d\in T^+$ and any $\alpha\in\Phi^+$ write $a$ as a finite sum of simple roots $\alpha=\sum\limits_{i}\alpha_i$. By definition now $\langle\alpha_i,\lambda_d\rangle\geq 0$ for every $\alpha_i\in\Delta^+$ and thus $\langle\alpha,\lambda_d\rangle\geq 0$. Using relation \eqref{actroot} gives: $$d^{-1}x_{-\alpha}(s\varpi)d=x_{-\alpha}(s'\varpi^{1+\langle\alpha,\lambda_d\rangle})$$ where $s,s'\in\mathcal{O}$ and thus $d^{-1}x_{-\alpha}(s\varpi)d\in J$ or equivalently $d^{-1}N_{\mathfrak{p}}^-d\subseteq J$. From the Iwahori factorization finally it follows that $JdJ=N_{\mathcal{O}}^+dd^{-1}T^0N_{\mathfrak{p}}^-dJ\subseteq N_{\mathcal{O}}^+dJ$.
\end{proof} 

We are now ready to show that:
\begin{theorem}[Diagonal Whittaker values] \label{diagon} ~\\
 Let $W\in \mathcal{W}(\pi_{St^{\chi}},\psi)^{J}$ with $W(1)=1$ be such that $W$ is an eigenfunction of $\mathcal{H}(G,J)$ with $$(F*W)(g)=\rho_{\chi}(F)W(g).$$
	Then $W$ has the following diagonal values:
$$W(d)=\begin{cases}
       (\chi\delta_B)(d)  & \text{if $d\in T^+$} \\
        0 & \text{otherwise} 
    \end{cases}.$$
 \end{theorem}
 \begin{proof}
Let $\epsilon_{g}=\rho_{\chi}(X_{g})\in\mathbb{C}$ for any $g\in G$. 
We want to compute $\epsilon_d$ for $d\in T^+$. To do that we will use the function $\phi_G^-\in St^{\chi}$ from Proposition \ref{fixvec}, since $\phi_G^-$ satisfies the same relation $F(\phi_G^-)(g)=\rho_{\chi}(F)\phi_G^-(g)$ as $W$ we get that $$\epsilon_h\phi_G^-(g)=X_h(\phi_G^-)(g), \text{for all } h\in G$$
and thus setting $h=d$ and $g=1$:
$$\epsilon_d\phi_G^-(1)=X_d(\phi_G^-)(1).$$
By Proposition \ref{fixvec}: $\phi_G^-(1)=(-q)^0f_1(1)=1$, so it remains to compute $X_d(\phi_G^-)(1)$. For this we notice that $X_d(f_w)(1)=\int_{JdJ}f_w(x)dx$ which from Lemma \ref{cosum1} and Lemma \ref{dreps} for $d\in T^+$ equals $\sum\limits_{\ell\in L}f_w(n_{\ell}d)$ which means that $X_d(f_w)(1)\ne 0$ if and only if $w=1$ and in that case $X_d(f_1)(1)=\sum\limits_{\ell\in L}(\chi\delta_B)(n_{\ell}d)=\sum\limits_{\ell\in L}(\chi\delta_B)(d)=\mu(JdJ)(\chi\delta_B)(d)$. Substituting in the original euation gives us: $$\frac{\epsilon_d}{\mu(JdJ)}=(\chi\delta_B)(d).$$ Now using Lemma \ref{cosum1} again for the action of $X_d$ on $W$ yields: $$\epsilon_dW(g)=X_d(W)(g)=\sum\limits_{\ell}W(gn_{\ell}d)$$ and thus since $W(1)=1$, for $g=1$ we obtain: $$\epsilon_d=\epsilon_d W(1)=\sum\limits_{\ell\in L}W(n_{\ell}d)=\sum\limits_{\ell\in L}\psi(n_{\ell})W(d)=\sum\limits_{\ell\in L}W(d)=\mu(JdJ)W(d).$$ Solving for $W(d)$ yields: $$W(d)=\frac{\epsilon_d}{\mu(JdJ)}=(\chi\delta_B)(d), \text{ for any } d\in T^+.$$
 \end{proof}
\begin{lemma} \label{diagcos}
  For any $w\in\mathbf{W}$ we have $Jd_wwJ=\underset{\ell\in L_w}{\bigcup}n_{\ell}^-d_wwJ$, where $n_{\ell}^-\in N^-_{\mathfrak{p}}$, $L_w$ is a finite set of indices and $|L_w|=\mu(Jd_wwJ)$.
\end{lemma}
\begin{proof}
    We write $J=N^-_{\mathfrak{p}}T^0N_{\mathcal{O}}$ as before. For any $n\in N_{\mathcal{O}}$ observe that $n=\prod\limits_{i=1}^rx_{a_i}(s_{a_i})$ for some $r\in\mathbb{N}$ with $a_i\in\Delta^+$ and $s_{a_i}\in \mathcal{O}$ for all indices. Since $w^{-1}d_w^{-1}nd_ww=\prod\limits_{i=1}^r(w^{-1}d_w^{-1}x_{a_i}(s_{a_i})d_ww)$ it suffices to show that $w^{-1}d_w^{-1}x_{\beta}(s_{\beta})d_ww=x_{w^{-1}(\beta)}{(\varpi^{-\langle\beta,\lambda_{d_w}\rangle}s'_{\beta})}\in J, \forall \beta\in\Delta^+$ and any $s_{\beta}\in\mathcal{O}$ and the corresponding $s_{\beta}'\in\mathcal{O}$. Dividing into two cases, we see that if $\beta\in\Delta^+:w^{-1}\beta\in\Phi^+$ then $\langle\beta,\lambda_{d_w}\rangle=0$ and similarly if $w^{-1}\beta\in\Phi^-$ then $\langle\beta,\lambda_{d_w}\rangle=-1$. Thus in every case
    $x_{w^{-1}(\beta)}{(\varpi^{-\langle\beta,\lambda_{d_w}\rangle}s_{\beta})}\in J$.
\end{proof}
\begin{lemma}\label{aux}
    For every Weyl element $w\in\mathbf{W}$ we have $d_{w_0w}=w_0d_ww_0d_{w_0}z_w$ for some $z_w\in\mathcal{Z}$.
\end{lemma}
\begin{proof}
 For all $\alpha\in\Delta^+$ we observe that $w_0\alpha\in\Delta^-$ and compute:  $\langle \alpha, \lambda_{d_{w_0}}+\lambda_{w_0d_ww_0}\rangle=\langle\alpha, \lambda_{d_{w_0}}\rangle+\langle\alpha, \lambda_{w_0d_ww_0}\rangle\overset{Remark   \ref{rem1}}{=}-1+\langle w_0\alpha,\lambda_{d_w}\rangle=-1-\langle -w_0\alpha,\lambda_{d_w}\rangle$. Assume that $w^{-1}w_0\alpha\in\Phi^-$ i.e. $w^{-1}(-w_0\alpha)\in\Phi^+$ and set $\beta=-w_0\alpha$, then $\beta\in\Delta^+$ and by definition $\langle \beta,\lambda_{d_w}\rangle=0$. Similarly when $w^{-1}w_0\alpha\in\Phi^+$ i.e. $w^{-1}(-w_0\alpha)\in\Phi^-$: $\langle -w_0\alpha,\lambda_{d_w}\rangle=-1$ thus $\langle \alpha, \lambda_{d_{w_0}}+\lambda_{w_0d_ww_0}\rangle=-1-(-1)=
      0  \text{ if } w^{-1}w_0\alpha\in\Phi^+$ and 
      $\langle\alpha,\lambda_{d_{w_0}}+\lambda_{w_0d_ww_0}\rangle=-1  \text{ if } w^{-1}w_0\alpha\in\Phi^-$ which is exactly $\langle \alpha, \lambda_{d_{w_0w}}\rangle$ by definition. This implies that $w_0d_ww_0d_{w_0}$ and $d_{w_0w}$ are equal up to a central element. 
\end{proof}
\begin{theorem}[Whittaker values at every cell $Tw$] \label{main} ~\\
Under the assumptions of Theorem \ref{diagon}
    $$W(dw)=\begin{cases}
      (\chi\delta_B)(d)(-q)^{-\ell(w)}  & \text{if $\lambda_d$ is $w$-dominant} \\
        0 & \text{otherwise} 
    \end{cases}$$
\end{theorem}
\begin{proof}
We use again the relation $X_{d_ww}(W)(g)=\epsilon_{d_ww}W(g)$. By Lemma \ref{diagcos} we have
$\epsilon_{d_ww}W(g)=\sum\limits_{\ell\in L_w}W(gn_{\ell}^-d_ww)$. We set $g=dd_{w_0}w_0(=dw_0d_{w_0}^{-1})$ and write $n_{\ell}^-=\prod\limits_{\alpha\in \Phi^-}x_{\alpha}(\varpi s_{\alpha})$ for some $s_{\alpha}\in\mathcal{O}$ depending on the indice $\ell$. We want to show that $\psi(gn_{\ell}^-g^{-1})=\psi(dw_0d_{w_0}^{-1}n_{\ell}^-d_{w_0}w_0d^{-1})=1$ for any $d\in T^+$. For this purpose we distinguish two cases for each $x_{\alpha}(\varpi s_{\alpha})$: 
\begin{itemize}
    \item for any $\alpha\in\Delta^-$: $w_0d_{w_0}^{-1}x_{\alpha}(\varpi s_{\alpha})d_{w_0}w_0=x_{w_0(\alpha)}(\varpi s_{\alpha}\varpi^{\langle\alpha,\lambda_{d_{w_0}^{-1}}\rangle})\in N_{\mathcal{O}}$ since $\langle\alpha,-\lambda_{d_{w_0}}\rangle={-1}$ by Remark \ref{domd}. Setting $n^+=w_0d_{w_0}^{-1}x_{\alpha}(\varpi s_{\alpha})d_{w_0}w_0$, it follows that $\psi(dn^+d^{-1})=1$ for any $d\in T^+$.
    \item for any $\alpha\in\Phi^-\backslash\Delta^-$ and $d\in T$: $\psi(dx_{w_0(\alpha)}(\varpi s_{\alpha}\varpi^{\langle\alpha,\lambda_{d_{w_0}^{-1}}\rangle})d^{-1})=1$ since $w_0\alpha\not\in\Delta$.
\end{itemize} For $g=dd_{w_0}w_0$ we then have that $W(gn_{\ell}^-g^{-1}gd_ww)=W(gd_ww)$, giving the relation:  $$\epsilon_{d_ww}W(dd_{w_0}w_0)=\mu(Jd_wwJ)W(dd_{w_0}w_0d_ww)=\mu(Jd_wwJ)W(dd_{w_0}w_0d_ww_0w_0w)$$ for any $d\in T^+$. Using Lemma \ref{aux} now, this expression becomes:
$$\epsilon_{d_ww}W(dd_{w_0}w_0)=\mu(Jd_wwJ)W(dd_{w_0w}w_0w),\text{ for all } d\in T.$$
Notice that the above relation is extended to all of $T$ and not just $d\in T^+$ since $d\not\in T^+$ means $dd_{w_0}\not\in T_{w_0}^+$ and $dd_{w_0w}\not\in T_{w_0w}^+$, so both sides are $0$.
Setting $w=w_0$ and comparing scalar terms with Lemma \ref{permlem} for $d=d_{w_0}^{-1}$ we get that $\epsilon_{d_{w_0}w_0}=(-q)^{\ell(w_0)}\mu(Jd_{w_0}w_0J)W(d_{w_0}^{-1})$. Substituting the value of $\epsilon_{d_{w_0}w_0}$ and setting $d\to dd_{w_0}^{-1}$ leads to: 
$$W(dw_0)=(-q)^{-\ell(w_0)}W(d_{w_0}^{-1})^{-1}W(dd_{w_0}^{-1}), \text{ for all } d\in T$$ Similarly we set $w\to w_0w$ and compute $\epsilon_{d_{w_0w}w_0w}=(-q)^{\ell(w_0)-\ell(w)}\mu(Jd_{ww_0}ww_0J)W(d_{w_0}^{-1}d_w)$. Substituting this value of $\epsilon_{d_{w_0w}w_0w}$ and using the above relation for the $W(dw_0)$ term implies:
$$W(dw)=(-q)^{-\ell(w)}W(d_{w}^{-1})^{-1}W(dd_{w}^{-1}), \text{ for all } d\in T \; \; (**)$$ Note that $d_w^{-1}\in T^+$ so a direct application of Theorem \ref{diagon} proves the claim since: 
\begin{itemize}
    \item if $d\not\in T_w^+$ both sides evaluate to zero,
    \item if $d\in T_w^+$ then it follows that $dd_w^{-1}\in T^+$, and the expression evaluates to $W(dw)=(-q)^{-\ell(w)}(\chi\delta_B)(d_w^{-1})^{-1}(\chi\delta_B)(dd_w^{-1})=(-q)^{-\ell(w)}(\chi\delta_B)(d)$.
\end{itemize}
\end{proof}

From the relation $(F*W)(g)=\rho_{\chi}(F)W(g)$ and the uniqueness in Theorem \ref{main} we conclude that: 
\begin{theorem} \label{StWhit}
    The function described in Theorem \ref{main} $$W(dw)=\begin{cases}
     (\chi\delta_B)(d)(-q)^{-\ell(w)}  & \text{if $\lambda_d$ is $w$-dominant} \\
        0 & \text{otherwise} 
    \end{cases}$$ is the unique normalized Whittaker function that corresponds to the Iwahori spherical vector $\phi_G^-\in (St^{\chi})^J$, under the condition $W(1)=1$.
\end{theorem}

\begin{remark}
Using this explicit description of the Whittaker function associated to $v\in St^{\chi}$, we remark that $v$ is not fixed by any other parahoric subgroup of $G$ since $q^{-1}=W(s_{\alpha})\ne W(1)=1$.
\end{remark}

\section{Split Reductive Groups}
In this section we extend the results to (connected) split reductive groups without the condition of a connected center. Let thus $\mathcal{G}$ be a general split reductive group. Assuming our representation has trivial central character, it suffices to study representations of $G/Z$. We set $\mathcal{G}'=\mathcal{G}/\mathcal{Z}$ the adjoint group of $\mathcal{G}$ and $G'=\mathcal{G}'$. Then the short exact sequence of group valued functors $$1\to \mathcal{Z}\to\mathcal{G}\to\mathcal{G}'\to 1$$ 
gives rise to a short exact sequence of groups:
$$1\to Z\to G\to G'\to H^1(F,\mathcal{Z})$$ where $H^1(F,\mathcal{Z})$ is the first cohomology group of $\mathcal{Z}$. Since $\mathcal{Z}$ is a central algebraic group this implies that $H^1(F,\mathcal{Z})$ is finite abelian. We have just proved that: 
\begin{lemma} \label{lem:irred}
    The group $G_1=G/Z$ is a finite index, normal subgroup of $G'$. Furthermore $G'/G_1$ is abelian.
\end{lemma}
Using the above lemma we are now in the context of Lemma 2.1 of \cite{gel-knap} which states that:
\begin{lemma}
   Let $\pi$ be an irreducible, smooth representation of $G'$, then:
   \begin{enumerate}
       \item The restriction $\pi|_{G_1}$ is a finite direct sum of irreducible representations of $G_1$.
       \item Grouping by equivalence classes $$\pi|_{G_1}\simeq\sum\limits_{i=1}^Mm_i\pi_i$$ with $\pi_i\ne\pi_j$ and irreducible, we have that $m_i=m_j=m$ for all indices. 
       \item The number of continuous characters $\chi$ of $G'$ trivial on $G_1$ in the set $$\mathbb{X}_{G_1}(\pi)=\{\chi:G'/G_1\to\mathbb{C}^{\times}|\chi\pi\simeq\pi\}$$ is $m^2M$.
   \end{enumerate}
\end{lemma}
Using the same notation as before, we have that:
\begin{proposition}
The restriction $\pi_{St^{\chi}}|_{G_1}$ remains irreducible as a representation of $G_1$.
\end{proposition}
\begin{proof}
    From 3. of Lemma \ref{lem:irred}, it suffices to show that $\mathbb{X}_{G_1}(\pi_{St^{\chi}})=\{1\}$. Let $\xi$ be a character of $G'$ that is trivial on $G_1$. Since $\mathbf{W}\subseteq G_1$, this implies that $\xi^w=\xi$ and thus the character $\xi\chi\delta_B$ is regular and obviously $\xi\otimes I(\chi)\simeq I(\xi\chi)$. But now assuming $\xi\otimes St^{\chi}\simeq St^{\chi}$ and comparing the Jacquet modules of both sides leads to $\xi\chi\delta_B=\chi\delta_B \implies \xi\equiv 1$ completing the proof.
\end{proof}
\begin{proposition}
        The restriction of the Whittaker function of Theorem \ref{StWhit} to $G_1$ is the unique up to scalar multiplication Iwahori spherical Whittaker function of $G_1$ corresponding to $(St^{\chi})^J$.
\end{proposition}
\begin{remark}
Let $\omega$ be a character of $G$ and assume that $\omega_T$, its restriction to the torus $T$, is trivial on $T^0$. Then since $\omega\otimes I(\chi)\simeq I(\omega_T\chi)$, it follows that that the formula of Theorem \ref{StWhit} also applies to $I(\omega_T\chi)$, with $\omega_T\chi$ replacing $\chi$. The trivial central character condition is thus not entirely restrictive.
\end{remark}
\begin{remark}
    In the case of a split reductive group $G$ with connected center $\mathcal{Z}$, the trivial central character condition can be removed.
\end{remark}
From the above we finally obtain that:
\begin{theorem}
    For a connected split reductive group $G$, the unique up to scalar Iwahori fixed Whittaker function corresponding to the Steinberg representation with trivial central character is the function in Theorem \ref{StWhit}. 
\end{theorem}

\section{A Rankin-Selberg computation in $GL_n$}
Throughout this section we let $G_n=GL_n(F)$ and using our results from Theorem \ref{StWhit}, we compute certain Rankin-Selberg integrals associated to Steinberg representations. Let $\pi_n=\pi_{St^1}$ be the Steinberg representation of $G_n$ and $W_n$ its' corresponding Iwahori fixed Whittaker function with unramified character $\psi$. Pick a multiplicative Haar measure $d^{\times}g$ on $G_m$ normalized to be $1$ on the maximal compact subgroup $K_m$. Denote by $N_m$ the unipotent subgroup of upper triangular matrices of $G_m$, by $T_m$ the diagonal torus and by $J_m$ the Iwahori subgroup of $K_m$.
For $\pi_n$ and $\pi_m$ with $n\geq m$ and $W\in\mathcal{W}(\pi_n,\psi_n), W'\in\mathcal{W}(\pi_m,\psi_m^{-1})$ where $\psi_m=\psi_n|_{N_m}$, the zeta functions are defined as:
$$Z(s,W,W')=\int_{N_m\backslash G_m}W(\begin{pmatrix}
    g & 0 \\
    0 & I_{n-m}
\end{pmatrix})W'(g)|det(g)|^{s-\frac{n-m}{2}}d^{\times}g$$
For $s$ sufficiently large, this integral converges absolutely and, varying $W$ and $W'$, gives a fractional ideal of $\mathbb{C}[q,q^{-1}]$ generated by the L-function $L(s,\pi_n\boxtimes\pi_m)$. For applications it is sometimes useful to find $W$ and $W'$ that give the $L$-function itself. These functions are commonly referred to as "test vectors". We now explicitly show that the function we computed in Theorem \ref{StWhit} is a test vector for the Steinberg representation by computing $Z(s,W_n,W_m)$. 
We start with the case of $m=1$ and take $W_1$ to be an unramified character $\chi(.)=|.|^{\alpha}$ for some $\alpha\in\mathbb{C}$.
\begin{theorem}[Test Vector for the Steinberg Representation]

    We have that: $$Z(s,W_n,\chi)=\int_{F^{\times}}W_n(\begin{pmatrix}
    g & 0 \\
    0 & I_{n-1}
\end{pmatrix})\chi(g)|g|^{s-\frac{n-1}{2}}d^{\times}g=\frac{1}{1-\chi(\varpi)q^{-s-\frac{n-1}{2}}}=L(s,\pi_n\times\chi).$$ (see also section 0, relation (8) and section 8, relation (5) of \cite{JacShaPi}).
\end{theorem}
\begin{proof}
    From Theorem \ref{StWhit} it is clear that $W_n(\begin{pmatrix}
    g & 0 \\
    0 & I_{n-1}
\end{pmatrix})=|g|^{n-1}\mathbbm{1}_{\mathcal{O}}(g)$. Plugging this formula of $W_n$ into the integral gives $Z(s,W_n,\chi)=\int_{\mathcal{O}^{\times}}\chi(g)|g|^{s+\frac{n-1}{2}}d^{\times}g=\sum\limits_{\ell=0}^{\infty}\int_{|g|=q^{-\ell}}\chi(g)|g|^{s+\frac{n-1}{2}}d^{\times}g=\sum\limits_{\ell=0}^{\infty}\chi(\varpi)^{\ell}q^{-\ell(s+\frac{n-1}{2})}=\frac{1}{1-\chi(\varpi)q^{-s-\frac{n-1}{2}}}$.
\end{proof}

We now continue with the other cases for which we need the following lemma: 
\begin{lemma} \label{varchange}
    For $f\in Ind_{N_m}^{G_m}1$ (compact induction) we have that $$\int_{N_m\backslash G_m}f(g)d^{\times}g=\sum\limits_{w\in\mathbf{W}_m}q^{\ell(w)}\int_{T_m\times J_m}f(twh)\delta_{B_m}(t)^{-1}d^{\times}tdh$$.
\end{lemma}
\begin{proof}
    From the Iwasawa decomposition and the properties of $d^{\times}g$ we have that $\int_{N_m\backslash G_m}f(g)d^{\times}g=\int_{T_m\times K_m}f(tk)\delta_{B_m}(t)^{-1}d^{\times}tdk$. From the Bruhat decomposition $K_m=\bigsqcup\limits_{w\in\mathbf{W}_m}J_mwJ_m$ and Lemma \ref{cosrep} we get $\int_{T_m\times K_m}f(tk)d^{\times}tdk=\sum\limits_{w\in\mathbf{W}_m}\sum\limits_{s_{\alpha}\in(\mathcal{O}/\mathfrak{p})^{\ell(w)}}\int_{T_m\times J_m}f(t\prod\limits_{\underset{w^{-1}\alpha\in\Phi^-}{\alpha\in\Phi^+|}}x_{\alpha}(s_{\alpha})wj)\delta_{B_m}(t)^{-1}d^{\times}tdj$. The trivial observation $tNt^{-1}\subseteq N$ then implies that $$\sum\limits_{w\in\mathbf{W}_m}\sum\limits_{s_{\alpha}\in(\mathcal{O}/\mathfrak{p})^{\ell(w)}}\int_{T_m\times J_m}f(twj)\delta_{B_m}(t)^{-1}d^{\times}tdj=\sum\limits_{w\in\mathbf{W}_m}q^{\ell(w)}\int_{T_m\times J_m}f(twj)\delta_{B_m}(t)^{-1}d^{\times}tdj.$$
\end{proof}

\begin{theorem}[Test Vector for Steinberg Pairs] 
    For $n\geq m$ we have the following expression: 
    $$Z(s,W_n,W_m)=\frac{q^{(\frac{m(m-1)}{2}(s+\frac{m+n}{2})-\frac{(m-1)m(m+1)}{3})}}{\prod\limits_{i=1}^m(1-q^{-s-\frac{n+m}{2}+i})}=q^{(\frac{m(m-1)}{2}(s+\frac{m+n}{2})-\frac{(m-1)m(m+1)}{3})}L(s,\pi_n\times\pi_m).$$
\end{theorem}
\begin{proof}
 Applying the relation of Lemma \ref{varchange} to the function $f(g)=W_n(\begin{pmatrix}
    g & 0 \\
    0 & I_{n-m}
\end{pmatrix})W_m(g)$ with $g\in N_m\backslash G_m$, we obtain $Z(s,\pi_n\boxtimes\pi_m)=\int_{N_m\backslash G_m}W_n(\begin{pmatrix}
    g & 0 \\
    0 & I_{n-m}
\end{pmatrix})W_m(g)|det(g)|^{s-\frac{n-m}{2}}d^{\times}g=\sum\limits_{w\in\mathbf{W}_m}q^{\ell(w)}\int_{J_m}\int_{T_m}W_n(\begin{pmatrix}
    twh & 0 \\
    0 & I_{n-m}
\end{pmatrix})W_m(twh)|det(t)|^{s-\frac{n-m}{2}}\delta_{B_m}^{-1}(t)d^{\times}tdh$, which by the left $J_m$-invariance of $W_n$ (with the natural embedding) and $W_m$ becomes:  $$Z(s,W_n,W_m)=\sum\limits_{w\in\mathbf{W}_m}q^{\ell(w)}\int_{T_m}W_n(\begin{pmatrix}
    tw & 0 \\
    0 & I_{n-m}
\end{pmatrix})W_m(tw)|det(t)|^{s-\frac{n-m}{2}}\delta_{B_m}^{-1}(t)d^{\times}t.$$ But now we note that $W_n(\begin{pmatrix}
    tw & 0 \\
    0 & I_{n-m}
\end{pmatrix})=(-q)^{-\ell(w)}\delta_{B_n}(\begin{pmatrix}
    t & 0 \\
    0 & I_{n-m}
\end{pmatrix})=(-q)^{-{\ell(w)}}\delta_{B_m}(t)|det(t)|^{n-m}$ precisely for $t\in T_w^+$. Similarly $W_m(tw)=(-q)^{-\ell(w)}\delta_{B_m}(t)$ for $t\in T_w^+$. This means that $$Z(s,W_n,W_m)=\sum\limits_{w\in\mathbf{W}_m}q^{-\ell(w)}\int_{T_w^+}\delta_{B_m}(t)|det(t)|^{s+\frac{n-m}{2}}d^{\times}t.$$
Setting explicitly $t=diag(t_1,\dots,t_m)$ the left hand side becomes $$\sum\limits_{w\in\mathbf{W}_m}q^{-\ell(w)}\prod\limits_{i=1}^m\int_{|t_i|\leq|t_{i+1}|q^{k_i(w)}}|t_1|^{m-1}|t_2|^{m-3}\dots |t_m|^{1-m}|t_1\dots t_m|^{s+\frac{n+m}{2}}\prod\limits_{i=1}^md^{\times}t_i$$ where 
$k_i(w)=\begin{cases}
    1 \text{ , if } w^{-1}(i+1)<w^{-1}(i) \\
    0 \text{ , otherwise}
\end{cases}.$ Setting $x_i=t_i/t_{i+1}$ with $x_m=t_m$ we note that $\delta_{B_m}(t)=\prod\limits_{i=1}^m|x_i|^{i(m-i)}$ and $det|t_1\dots t_m|=\prod\limits_{i=1}^m|x_i|^i$, making the integral:
$$\sum\limits_{w\in\mathbf{W}_m}q^{-\ell(w)}\prod\limits_{i=1}^m\int_{|x_i|\leq q^{k_i(w)}}|x_i|^{i(m-i+s+\frac{n-m}{2})}d^{\times}x_i$$ and each term evaluates to $$\int_{|x_i|\leq q^{-k_i(w)}}|x_i|^{i(m-i+s+\frac{n+m}{2})}d^{\times}x_i=\sum\limits_{j=k_i(w)}^{\infty}q^{-ji(s-i+\frac{n-m}{2})}=\frac{q^{k_i(w)i(s-i+\frac{n-m}{2})}}{1-q^{-i(s-i+\frac{n-m}{2})}}$$ giving us the expression for the $Z$ function: 
$$Z(s,W_n,W_m)=\prod\limits_{i=1}^m(1-q^{-i(s-i+\frac{n-m}{2})})^{-1}\sum\limits_{w\in\mathbf{W}_m}q^{-\ell(w)}\prod\limits_{i=1}^{m-1}q^{k_i(w)i(s-i+\frac{n-m}{2})}.$$
To simplify this expression set $X=q^{s+\frac{n+m}{2}}$ to obtain the expression $$\sum\limits_{w\in\mathbf{W}_m}q^{-\ell(w)}\prod\limits_{i=1}^{m-1}X^{k_i(w)i}q^{-k_i(w)i^2}=\sum\limits_{w\in\mathbf{W}_m}q^{-\ell(w)}\prod\limits_{\underset{w^{-1}(i)>w^{-1}(i+1)}{i\in\{1,...,m-1\}:}}X^iq^{-i^2}$$ and notice that by using the standard combinatorial notations as in \cite{BrightSavage2010}:
\begin{itemize}
    \item $\ell(w)=inv(w)$ is the total number of inversions of $w\in S_m$
    \item $\prod\limits_{\underset{w^{-1}(i)>w^{-1}(i+1)}{i\in\{1,...,m-1\}:}}X^i=\prod\limits_{i\in Des(w)}X^i=X^{\sum\limits_{i\in Des(w)}i}=X^{maj(w)}$ and 
    \item $\prod_{i\in Des(w)}q^{-i^2}=q^{-sq(w)}$
\end{itemize}
the expression is exactly  $\sum\limits_{w\in\mathbf{W}_m}q^{-(inv(w)+sq(w))}X^{maj(w)}$, which by Theorem 4.4 of \cite{BrightSavage2010} is equal to:
$$\prod_{i=1}^m\frac{1-X^iq^{-i^2}}{1-Xq^{-i}}.$$
Substituting in the full expression yields that 
$$Z(s,W_m,W_m)=\frac{\prod_{i=1}^m\frac{(1-X^iq^{-i^2})}{(1-Xq^{-i})}}{\prod_{i=1}^m(1-X^{-i}q^{i^2})}=\frac{\prod\limits_{i=1}^mX^{i-1}q^{-i(i-1)}}{\prod\limits_{i=1}^m(1-X^{-1}q^i)}=\frac{X^{\frac{m(m-1)}{2}}q^{-\frac{(m-1)m(m+1)}{3}}}{\prod\limits_{i=1}^m(1-X^{-1}q^i)}.$$ Substituting the value of $X$ back into this expression proves the claim.
\end{proof}


\begin{thebibliography}{99}

\bibitem{bor}
A. Borel, Admissible representations of a semi-simple group over a local field with vectors fixed under an Iwahori subgroup, Inv. Math., 35, 233- 259 (1976)

\bibitem{BrightSavage2010}
K. L. Bright, C.D. Savage,
The geometry of lecture hall partitions and quadratic permutation statistics,
Proceedings of the 22nd International Conference on Formal Power Series and Algebraic Combinatorics,
Discrete Mathematics \& Theoretical Computer Science Proceedings AN, pp.569--580 (2010)

\bibitem{Bump}
B. Brubaker, V. Buciumas, D. Bump, Henrik P. A. Gustafsson, Colored Vertex Models and Iwahori Whittaker Functions, 2019, \href{https://arxiv.org/abs/1906.04140}{arXiv: 1906.04140 [math.RT]}

\bibitem{cassnotes}
W. Casselman, An introduction to the theory of admissible representations of reductive p-adic groups, Unpublished notes

\bibitem{cass2}
W. Casselman, The unramified principal series of p-adic groups. I. The spherical function, Comp. Math. 40, No 3 (1980), p. 387-406


\bibitem{cur}
C. W. Curtis, N. Iwahori, and R. Kilmoyer, Hecke algebras and characters of parabolic type of finite groups with $(B,N)-$pairs, Inst. Hautes Études Sci. Publ. Math. 40 (1972), 81-116


\bibitem{gel-knap}
S. S. Gelbart, and A. W. Knapp, {L}-indistinguishability and {R} Groups for the Special Linear Group, Adv. Math. 43 (1982), p. 101-121


\bibitem{iwmats}
N. Iwahori, H. Matsumoto, On some Bruhat decomposition and the structure of the Hecke rings of $p$-adic Chevalley groups,
Publications Mathématiques de l’I.H.É.S.,  tome 25 (1965), p. 5-48



\bibitem{JacShaPi}
H. Jacquet, I. I. Piatetskii-Shapiro, J. A. Shalika, Rankin-Selberg Convolutions, American Journal of Mathematics, Vol. 105, No. 2 (Apr., 1983), pp. 367-464





\bibitem{reeder}
M. Reeder, $p$-adic Whittaker functions and vector bundles on flag manifolds, Comp. Math. 85, No 1 (1993), p. 9 - 36


\bibitem{shal}
J. A. Shalika, On the space of cusp forms of a p-adic Chevalley group, Ann. Math. , Second Series, 92 (2) p. 262–278 (1970)





\bibitem{shint}
T. Shintani, On an Explicit Formula for Class-$1$ "Whittaker Functions" on $GL_n$ over $p$-adic Fields, Proc. Japan Acad. 52, (1976), p. 180 - 182








\end{thebibliography}
\end{document}